\documentclass[12pt]{article}
\usepackage{amsmath,amssymb,amsthm}
\usepackage{tikz}
\usetikzlibrary{calc}

\def\quintetspic{\begin{tikzpicture}
\tikzstyle{knode}=[circle,draw=black,thick,inner sep=0pt,text width = 15 pt, align = center]

\node (1) at (-5,2) [knode] {\small $g_1$};
\node (2) at (-3,2) [knode] {\small $h_1$};
\node (3) at (-5,0) [knode] {\small $g_2$};
\node (4) at (-3,0) [knode] {\small $h_2$};
\node (5) at (-4,-1) [knode] {\small $s$};

\node (6) at (-1,2) [knode] {\small $g_1$};
\node (7) at (1,2) [knode] {\small $h_1$};
\node (8) at (-1,0) [knode] {\small $g_2$};
\node (9) at (1,0) [knode] {\small $h_2$};
\node (10) at (0,-1) [knode] {\small $s$};

\node (11) at (3,2) [knode] {\small $g_1$};
\node (12) at (5,2) [knode] {\small $h_1$};
\node (13) at (3,0) [knode] {\small $h_2$};
\node (14) at (5,0) [knode] {\small $g_2$};
\node (15) at (4,-1) [knode] {\small $s$};

\foreach \from/\to in {1/2,1/3,1/4,2/3,2/4,3/4,5/1,5/3,6/7,6/8,7/9,8/9,10/8,10/6,11/12,11/13,12/14,13/14,15/11,15/14}
\draw (\from)--(\to);
\end{tikzpicture}}

\def\tenspic{\begin{tikzpicture}[scale=.6]
\tikzstyle{knode}=[circle,draw=black,thick,inner sep=0pt,text width = 14 pt, align = center]

\node (g0) at (0:-5cm) [knode] {\small 0};
\node (g1) at ($(g0)+(0:2cm)$) [knode] {\small 1};
\node (g2) at ($(g0)+(180:2cm)$) [knode] {\small 2};
\node (g3) at ($(g0)+(120:2cm)$) [knode] {\small 3};
\node (g4) at ($(g3)+(180:2cm)$) [knode] {\small 4};
\node (g5) at ($(g0)+(60:2cm)$) [knode] {\small 5};
\node (g9) at ($(g0)+(180:4cm)$) [knode] {\small 9};
\node (g6) at ($(g2)+(-60:2cm)$) [knode] {\small 6};
\node (g7) at ($(g1)+(60:2cm)$) [knode] {\small 7};
\node (g8) at ($(g5)+(90:2cm)$) [knode] {\small 8};

\foreach \from/\to in {g0/g1,g0/g5,g0/g3,g0/g2,g1/g5,g1/g7,g2/g3,g2/g4,g2/g9,g4/g9,g5/g8,g6/g9}
\draw (\from)--(\to);

\foreach \from/\to in {g5/g7,g3/g5,g3/g4,g2/g6,g3/g6}
\draw (\from)--(\to);

\node (h0) at (0:5cm) [knode] {\small0};
\node (h1) at ($(h0)+(0:2cm)$) [knode] {\small1};
\node (h2) at ($(h0)+(180:2cm)$) [knode] {\small 2};
\node (h3) at ($(h0)+(120:2cm)$) [knode] {\small 3};
\node (h4) at ($(h3)+(180:2cm)$) [knode] {\small4};
\node (h5) at ($(h0)+(60:2cm)$) [knode] {\small 5};
\node (h9) at ($(h0)+(180:4cm)$) [knode] {\small 9};
\node (h6) at ($(h2)+(-60:2cm)$) [knode] {\small 6};
\node (h7) at ($(h1)+(60:2cm)$) [knode] {\small 7};
\node (h8) at ($(h5)+(90:2cm)$) [knode] {\small 8};

\foreach \from/\to in {h0/h1,h0/h5,h0/h3,h0/h2,h1/h5,h1/h7,h2/h3,h2/h4,h2/h9,h4/h9,h5/h8,h6/h9}
\draw (\from)--(\to);

\foreach \from/\to in {h7/h8,h3/h8,h2/h5,h2/h6}
\draw (\from)--(\to);

\end{tikzpicture}
}

\def\sevens{\begin{tabular}{c|c|c}
\begin{tikzpicture}[scale=0.6]
\tikzstyle{knode}=[circle,draw=black,thick,inner sep=2pt]

\node (lc) at (-1.5,0) {};
\node (rc) at (1.5,0) {};

\node (l1) at ($(lc)+(-1,3)$) [knode] {};
\node (l2) at ($(lc)+(1,3)$) [knode] {};
\node (l3) at ($(lc)+(-1,1)$) [knode] {};
\node (l4) at ($(lc)+(1,1)$) [knode] {};
\node (l5) at ($(lc)+(-1,-1)$) [knode] {};
\node (l6) at ($(lc)+(1,-1)$) [knode] {};
\node (l7) at ($(lc)+(0,-2.5)$) [knode] {};

\node (r1) at ($(rc)+(-1,3)$) [knode] {};
\node (r2) at ($(rc)+(1,3)$) [knode] {};
\node (r3) at ($(rc)+(-1,1)$) [knode] {};
\node (r4) at ($(rc)+(1,1)$) [knode] {};
\node (r5) at ($(rc)+(-1,-1)$) [knode] {};
\node (r6) at ($(rc)+(1,-1)$) [knode] {};
\node (r7) at ($(rc)+(0,-2.5)$) [knode] {};

\foreach \from/\to in {1/2,1/3,1/6,1/7,2/3,2/7,2/4,3/4,3/6,3/5,4/5,4/6,5/6}
\draw (l\from)--(l\to) 
(r\from)--(r\to); 

\draw (l7)--(l3)
(l7)--(l5)
(r4)--(r7)
(r6)--(r7);
\end{tikzpicture}

&

\begin{tikzpicture}[scale=0.6]
\tikzstyle{knode}=[circle,draw=black,thick,inner sep=2pt]

\node (t) at (0,3.2) {};

\node (lc) at (-1.5,0) {};
\node (rc) at (1.5,0) {};

\node (l1) at ($(lc)+(-1,3)$) [knode] {};
\node (l2) at ($(lc)+(1,3)$) [knode] {};
\node (l3) at ($(lc)+(-1,1)$) [knode] {};
\node (l4) at ($(lc)+(1,1)$) [knode] {};
\node (l5) at ($(lc)+(-1,-1)$) [knode] {};
\node (l6) at ($(lc)+(1,-1)$) [knode] {};
\node (l7) at ($(lc)+(0,-2.5)$) [knode] {};

\node (r1) at ($(rc)+(-1,3)$) [knode] {};
\node (r2) at ($(rc)+(1,3)$) [knode] {};
\node (r3) at ($(rc)+(-1,1)$) [knode] {};
\node (r4) at ($(rc)+(1,1)$) [knode] {};
\node (r5) at ($(rc)+(-1,-1)$) [knode] {};
\node (r6) at ($(rc)+(1,-1)$) [knode] {};
\node (r7) at ($(rc)+(0,-2.5)$) [knode] {};

\foreach \from/\to in {1/3,1/4,2/5,2/4,3/4,3/5,3/6,4/5,4/6,5/6}
\draw (l\from)--(l\to) 
(r\from)--(r\to); 

\draw (l7)--(l3)
(l7)--(l5)
(r4)--(r7)
(r6)--(r7);

\end{tikzpicture}

&
\begin{tikzpicture}[scale=0.6]
\tikzstyle{knode}=[circle,draw=black,thick,inner sep=2pt]

\node (lc) at (-1.5,0) {};
\node (rc) at (1.5,0) {};

\node (l1) at ($(lc)+(-1,3)$) [knode] {};
\node (l2) at ($(lc)+(1,3)$) [knode] {};
\node (l3) at ($(lc)+(-1,1)$) [knode] {};
\node (l4) at ($(lc)+(1,1)$) [knode] {};
\node (l5) at ($(lc)+(-1,-1)$) [knode] {};
\node (l6) at ($(lc)+(1,-1)$) [knode] {};
\node (l7) at ($(lc)+(0,-2.5)$) [knode] {};

\node (r1) at ($(rc)+(-1,3)$) [knode] {};
\node (r2) at ($(rc)+(1,3)$) [knode] {};
\node (r3) at ($(rc)+(-1,1)$) [knode] {};
\node (r4) at ($(rc)+(1,1)$) [knode] {};
\node (r5) at ($(rc)+(-1,-1)$) [knode] {};
\node (r6) at ($(rc)+(1,-1)$) [knode] {};
\node (r7) at ($(rc)+(0,-2.5)$) [knode] {};

\foreach \from/\to in {1/3,1/6,1/7,2/3,2/7,2/4,3/4,3/5,3/6,4/5,4/6,5/6}
\draw (l\from)--(l\to) 
(r\from)--(r\to);

\draw (l7)--(l3)
(l7)--(l5)
(r4)--(r7)
(r6)--(r7);

\end{tikzpicture}
\\

\hline

\begin{tikzpicture}[scale=0.6]
\tikzstyle{knode}=[circle,draw=black,thick,inner sep=2pt]

\node (t) at (0,3.2) {};

\node (lc) at (-1.5,0) {};
\node (rc) at (1.5,0) {};

\node (l1) at ($(lc)+(-1,3)$) [knode] {};
\node (l2) at ($(lc)+(1,3)$) [knode] {};
\node (l3) at ($(lc)+(-1,1)$) [knode] {};
\node (l4) at ($(lc)+(1,1)$) [knode] {};
\node (l5) at ($(lc)+(-1,-1)$) [knode] {};
\node (l6) at ($(lc)+(1,-1)$) [knode] {};
\node (l7) at ($(lc)+(0,-2.5)$) [knode] {};

\node (r1) at ($(rc)+(-1,3)$) [knode] {};
\node (r2) at ($(rc)+(1,3)$) [knode] {};
\node (r3) at ($(rc)+(-1,1)$) [knode] {};
\node (r4) at ($(rc)+(1,1)$) [knode] {};
\node (r5) at ($(rc)+(-1,-1)$) [knode] {};
\node (r6) at ($(rc)+(1,-1)$) [knode] {};
\node (r7) at ($(rc)+(0,-2.5)$) [knode] {};

\foreach \from/\to in {1/3,1/4,2/5,2/4,1/7,3/4,3/5,3/6,4/5,4/6,5/6}
\draw (l\from)--(l\to) 
(r\from)--(r\to);

\draw (l7)--(l3)
(l7)--(l5)
(r4)--(r7)
(r6)--(r7);

\end{tikzpicture}

&

\begin{tikzpicture}[scale=0.6]
\tikzstyle{knode}=[circle,draw=black,thick,inner sep=2pt]

\node (lc) at (-1.5,0) {};
\node (rc) at (1.5,0) {};

\node (l1) at ($(lc)+(-1,3)$) [knode] {};
\node (l2) at ($(lc)+(1,3)$) [knode] {};
\node (l3) at ($(lc)+(-1,1)$) [knode] {};
\node (l4) at ($(lc)+(1,1)$) [knode] {};
\node (l5) at ($(lc)+(-1,-1)$) [knode] {};
\node (l6) at ($(lc)+(1,-1)$) [knode] {};
\node (l7) at ($(lc)+(0,-2.5)$) [knode] {};

\node (r1) at ($(rc)+(-1,3)$) [knode] {};
\node (r2) at ($(rc)+(1,3)$) [knode] {};
\node (r3) at ($(rc)+(-1,1)$) [knode] {};
\node (r4) at ($(rc)+(1,1)$) [knode] {};
\node (r5) at ($(rc)+(-1,-1)$) [knode] {};
\node (r6) at ($(rc)+(1,-1)$) [knode] {};
\node (r7) at ($(rc)+(0,-2.5)$) [knode] {};

\foreach \from/\to in {1/2,1/3,1/6,2/3,2/4,3/4,3/5,4/6,5/6}
\draw (l\from)--(l\to) 
(r\from)--(r\to);

\draw (l7)--(l3)
(l7)--(l5)
(r4)--(r7)
(r6)--(r7);

\end{tikzpicture}

&

\begin{tikzpicture}[scale=0.6]
\tikzstyle{knode}=[circle,draw=black,thick,inner sep=2pt]

\node (lc) at (-1.5,0) {};
\node (rc) at (1.5,0) {};

\node (l1) at ($(lc)+(-1,3)$) [knode] {};
\node (l2) at ($(lc)+(1,3)$) [knode] {};
\node (l3) at ($(lc)+(-1,1)$) [knode] {};
\node (l4) at ($(lc)+(1,1)$) [knode] {};
\node (l5) at ($(lc)+(-1,-1)$) [knode] {};
\node (l6) at ($(lc)+(1,-1)$) [knode] {};
\node (l7) at ($(lc)+(0,-2.5)$) [knode] {};

\node (r1) at ($(rc)+(-1,3)$) [knode] {};
\node (r2) at ($(rc)+(1,3)$) [knode] {};
\node (r3) at ($(rc)+(-1,1)$) [knode] {};
\node (r4) at ($(rc)+(1,1)$) [knode] {};
\node (r5) at ($(rc)+(-1,-1)$) [knode] {};
\node (r6) at ($(rc)+(1,-1)$) [knode] {};
\node (r7) at ($(rc)+(0,-2.5)$) [knode] {};

\foreach \from/\to in {1/4,2/3,2/4,3/4,3/5,4/6,5/6}
\draw (l\from)--(l\to) 
(r\from)--(r\to);

\draw (l1) to [bend right] (l5);
\draw (r1) to [bend right] (r5);

\draw (l7)--(l3)
(l7)--(l5)
(r4)--(r7)
(r6)--(r7);

\end{tikzpicture}

\\
\hline 

\begin{tikzpicture}[scale=0.6]
\tikzstyle{knode}=[circle,draw=black,thick,inner sep=2pt]

\node (t) at (0,3.2) {};

\node (lc) at (-1.5,0) {};
\node (rc) at (1.5,0) {};

\node (l1) at ($(lc)+(-1,3)$) [knode] {};
\node (l2) at ($(lc)+(1,3)$) [knode] {};
\node (l3) at ($(lc)+(-1,1)$) [knode] {};
\node (l4) at ($(lc)+(1,1)$) [knode] {};
\node (l5) at ($(lc)+(-1,-1)$) [knode] {};
\node (l6) at ($(lc)+(1,-1)$) [knode] {};
\node (l7) at ($(lc)+(0,-2.5)$) [knode] {};

\node (r1) at ($(rc)+(-1,3)$) [knode] {};
\node (r2) at ($(rc)+(1,3)$) [knode] {};
\node (r3) at ($(rc)+(-1,1)$) [knode] {};
\node (r4) at ($(rc)+(1,1)$) [knode] {};
\node (r5) at ($(rc)+(-1,-1)$) [knode] {};
\node (r6) at ($(rc)+(1,-1)$) [knode] {};
\node (r7) at ($(rc)+(0,-2.5)$) [knode] {};

\foreach \from/\to in {1/6,1/3,2/3,2/4,3/4,3/5,4/6,5/6,2/7}
\draw (l\from)--(l\to) 
(r\from)--(r\to);

\draw (l7)--(l3)
(l7)--(l5)
(r4)--(r7)
(r6)--(r7);

\end{tikzpicture}

&

\begin{tikzpicture}[scale=0.6]
\tikzstyle{knode}=[circle,draw=black,thick,inner sep=2pt]

\node (lc) at (-1.5,0) {};
\node (rc) at (1.5,0) {};

\node (l1) at ($(lc)+(-1,3)$) [knode] {};
\node (l2) at ($(lc)+(1,3)$) [knode] {};
\node (l3) at ($(lc)+(-1,1)$) [knode] {};
\node (l4) at ($(lc)+(1,1)$) [knode] {};
\node (l5) at ($(lc)+(-1,-1)$) [knode] {};
\node (l6) at ($(lc)+(1,-1)$) [knode] {};
\node (l7) at ($(lc)+(0,-2.5)$) [knode] {};

\node (r1) at ($(rc)+(-1,3)$) [knode] {};
\node (r2) at ($(rc)+(1,3)$) [knode] {};
\node (r3) at ($(rc)+(-1,1)$) [knode] {};
\node (r4) at ($(rc)+(1,1)$) [knode] {};
\node (r5) at ($(rc)+(-1,-1)$) [knode] {};
\node (r6) at ($(rc)+(1,-1)$) [knode] {};
\node (r7) at ($(rc)+(0,-2.5)$) [knode] {};

\foreach \from/\to in {1/2,1/3,1/6,1/7,2/3,2/4,2/7,3/4,3/5,4/6,5/6}
\draw (l\from)--(l\to) 
(r\from)--(r\to);

\draw (l7)--(l3)
(l7)--(l5)
(r4)--(r7)
(r6)--(r7);

\end{tikzpicture}

&

\begin{tikzpicture}[scale=0.6]
\tikzstyle{knode}=[circle,draw=black,thick,inner sep=2pt]

\node (lc) at (-1.5,0) {};
\node (rc) at (1.5,0) {};

\node (l1) at ($(lc)+(-1,3)$) [knode] {};
\node (l2) at ($(lc)+(1,3)$) [knode] {};
\node (l3) at ($(lc)+(-1,1)$) [knode] {};
\node (l4) at ($(lc)+(1,1)$) [knode] {};
\node (l5) at ($(lc)+(-1,-1)$) [knode] {};
\node (l6) at ($(lc)+(1,-1)$) [knode] {};
\node (l7) at ($(lc)+(0,-2.5)$) [knode] {};

\node (r1) at ($(rc)+(-1,3)$) [knode] {};
\node (r2) at ($(rc)+(1,3)$) [knode] {};
\node (r3) at ($(rc)+(-1,1)$) [knode] {};
\node (r4) at ($(rc)+(1,1)$) [knode] {};
\node (r5) at ($(rc)+(-1,-1)$) [knode] {};
\node (r6) at ($(rc)+(1,-1)$) [knode] {};
\node (r7) at ($(rc)+(0,-2.5)$) [knode] {};

\foreach \from/\to in {1/2,1/3,1/4,2/7,3/4,3/5,4/6,5/6}
\draw (l\from)--(l\to) 
(r\from)--(r\to);

\draw (l2) to [bend left] (l6);
\draw (r2) to [bend left] (r6);

\draw (l7)--(l3)
(l7)--(l5)
(r4)--(r7)
(r6)--(r7);

\end{tikzpicture}

\\
\hline
\begin{tikzpicture}[scale=0.6]
\tikzstyle{knode}=[circle,draw=black,thick,inner sep=2pt]

\node (t) at (0,3.2) {};

\node (lc) at (-1.5,0) {};
\node (rc) at (1.5,0) {};

\node (l1) at ($(lc)+(-1,3)$) [knode] {};
\node (l2) at ($(lc)+(1,3)$) [knode] {};
\node (l3) at ($(lc)+(-1,1)$) [knode] {};
\node (l4) at ($(lc)+(1,1)$) [knode] {};
\node (l5) at ($(lc)+(-1,-1)$) [knode] {};
\node (l6) at ($(lc)+(1,-1)$) [knode] {};
\node (l7) at ($(lc)+(0,-2.5)$) [knode] {};

\node (r1) at ($(rc)+(-1,3)$) [knode] {};
\node (r2) at ($(rc)+(1,3)$) [knode] {};
\node (r3) at ($(rc)+(-1,1)$) [knode] {};
\node (r4) at ($(rc)+(1,1)$) [knode] {};
\node (r5) at ($(rc)+(-1,-1)$) [knode] {};
\node (r6) at ($(rc)+(1,-1)$) [knode] {};
\node (r7) at ($(rc)+(0,-2.5)$) [knode] {};

\foreach \from/\to in {1/3,1/2,2/3,2/4,2/7,3/4,3/5,4/6,5/6}
\draw (l\from)--(l\to) 
(r\from)--(r\to);

\draw (l1) to [bend right] (l5);
\draw (r1) to [bend right] (r5);

\draw (l3)--(l7)
(l6)--(l7)
(r4)--(r7)
(r5)--(r7);

\end{tikzpicture}

&
\begin{tikzpicture}[scale=0.6]
\tikzstyle{knode}=[circle,draw=black,thick,inner sep=2pt]

\node (lc) at (-1.5,0) {};
\node (rc) at (1.5,0) {};

\node (l1) at ($(lc)+(-1,3)$) [knode] {};
\node (l2) at ($(lc)+(1,3)$) [knode] {};
\node (l3) at ($(lc)+(-1,1)$) [knode] {};
\node (l4) at ($(lc)+(1,1)$) [knode] {};
\node (l5) at ($(lc)+(-1,-1)$) [knode] {};
\node (l6) at ($(lc)+(1,-1)$) [knode] {};
\node (l7) at ($(lc)+(0,-2.5)$) [knode] {};

\node (r1) at ($(rc)+(-1,3)$) [knode] {};
\node (r2) at ($(rc)+(1,3)$) [knode] {};
\node (r3) at ($(rc)+(-1,1)$) [knode] {};
\node (r4) at ($(rc)+(1,1)$) [knode] {};
\node (r5) at ($(rc)+(-1,-1)$) [knode] {};
\node (r6) at ($(rc)+(1,-1)$) [knode] {};
\node (r7) at ($(rc)+(0,-2.5)$) [knode] {};

\foreach \from/\to in {1/2,1/3,2/3,2/4,3/4,2/7,3/5,4/6,5/6}
\draw (l\from)--(l\to) 
(r\from)--(r\to);

\draw (l3)--(l7)
(l6)--(l7)
(r4)--(r7)
(r5)--(r7);

\end{tikzpicture}

\end{tabular}}

\def\othertens{\begin{tikzpicture}[scale=0.8]
\tikzstyle{knode}=[circle,draw=black,thick,inner sep=3pt]

\node (ltc) at (-2.5,0) {};
\node (rtc) at (2.5,0) {};

\node (l0) at ($(ltc)+(11.79:2cm)$) [knode] {};
\node (l1) at ($(ltc)+(63.93:2cm)$) [knode] {};
\node (l2) at ($(ltc)+(116.07:2cm)$) [knode] {};
\node (l3) at ($(ltc)+(168.21:2cm)$) [knode] {};
\node (l4) at ($(ltc)+(220.35:2cm)$) [knode] {};
\node (l5) at ($(ltc)+(272.49:2cm)$) [knode] {};
\node (l6) at ($(ltc)+(324.63:2cm)$) [knode] {};

\node (l7) at ($(ltc)+(220.35:3cm)$) [knode] {};
\node (l8) at ($(ltc)+(272.49:3cm)$) [knode] {};
\node (l9) at ($(ltc)+(324.63:3cm)$) [knode] {};

  \foreach \from/\to in {l0/l2,l0/l3,l0/l4,l0/l5,l0/l6,l1/l2,l1/l3,l1/l4,l1/l5,l1/l6,l2/l3,l2/l4,l2/l5,l2/l6,l3/l4,l3/l5,l3/l6,l4/l5,l4/l6,l5/l6,l7/l4,l8/l5,l9/l6} 
  \draw (\from)--(\to); 

\node (r0) at ($(rtc)+(0:2cm)$) [knode] {};
\node (r1) at ($(rtc)+(45:2cm)$) [knode] {};
\node (r2) at ($(rtc)+(90:2cm)$) [knode] {};
\node (r3) at ($(rtc)+(135:2cm)$) [knode] {};
\node (r4) at ($(rtc)+(180:2cm)$) [knode] {};
\node (r5) at ($(rtc)+(225:2cm)$) [knode] {};
\node (r6) at ($(rtc)+(270:2cm)$) [knode] {};
\node (r7) at ($(rtc)+(315:2cm)$) [knode] {};
\node (r8) at ($(rtc)+(0:0cm)$) [knode] {};

\node (r9) at ($(rtc)+(247.5:3cm)$) [knode] {};

  \foreach \from/\to in {r0/r1,r1/r2,r2/r3,r3/r4,r4/r5,r5/r6,r6/r7,r7/r0,r0/r8,r1/r8,r2/r8,r3/r8,r4/r8,r5/r8,r6/r8,r7/r8,r8/r9} 
  \draw (\from)--(\to);

\end{tikzpicture}}

\def\othereights{\begin{tikzpicture}[scale=0.8]
\tikzstyle{knode}=[circle,draw=black,thick,inner sep=3pt]

\node (lc) at (-2.5,0) {};
\node (rc) at (2.5,0) {};

\node (l0) at ($(lc)+(-2,1)$) [knode] {};
\node (l1) at ($(lc)+(-2,0)$) [knode] {};
\node (l2) at ($(lc)+(-2,-1)$) [knode] {};
\node (l3) at ($(lc)+(-1,0)$) [knode] {};
\node (l4) at ($(lc)+(0,1)$) [knode] {};
\node (l5) at ($(lc)+(0,-1)$) [knode] {};
\node (l6) at ($(lc)+(1,0)$) [knode] {};
\node (l7) at ($(lc)+(2,0)$) [knode] {};

\foreach \from/\to in {l0/l3,l1/l3,l2/l3,l3/l4,l3/l5,l4/l6,l5/l6,l6/l7}
\draw (\from)--(\to);

\node (r0) at ($(rc)+(-2,1)$) [knode] {};
\node (r1) at ($(rc)+(-2,-1)$) [knode] {};
\node (r2) at ($(rc)+(-1,0)$) [knode] {};
\node (r3) at ($(rc)+(0,0)$) [knode] {};
\node (r4) at ($(rc)+(1,1)$) [knode] {};
\node (r5) at ($(rc)+(1,0)$) [knode] {};
\node (r6) at ($(rc)+(1,-1)$) [knode] {};
\node (r7) at ($(rc)+(2,0)$) [knode] {};

\foreach \from/\to in {r0/r2,r1/r2,r2/r3,r3/r4,r3/r5,r3/r6,r4/r5,r5/r6,r4/r7,r5/r7,r6/r7}
\draw (\from)--(\to);
\end{tikzpicture}}

\renewcommand{\phi}{\varphi}

\theoremstyle{plain}

\newtheorem{theorem}{Theorem}
\newtheorem{corollary}{Corollary}

\theoremstyle{definition}

\allowdisplaybreaks

\begin{document}
\title{A construction of distance cospectral graphs}
\author{Kristin Heysse\thanks{Dept.\ of Mathematics, Iowa State University, Ames, IA 50011, USA\newline ({\tt keheysse@iastate.edu})}}

\maketitle



\begin{abstract}
The distance matrix of a connected graph is the symmetric matrix with columns and rows indexed by the vertices and entries that are the pairwise distances between the corresponding vertices. We give a construction for graphs which differ in their edge counts yet are cospectral with respect to the distance matrix. Further, we identify a subgraph switching behavior which constructs additional distance cospectral graphs. The proofs for both constructions rely on a perturbation of (most of) the distance eigenvectors of one graph to yield the distance eigenvectors of the other.
\end{abstract}



\section{Introduction}\label{sec:introduction}
Spectral graph theory explores the relationship between a graph and the eigenvalues (i.e., spectrum) of a matrix associated with that graph. There are a handful of common ways to associate a matrix to a graph, and the spectrum of each matrix holds a variety of information about the graph (see \cite{BH}). However, each matrix also has limitations in what information its spectrum can contain. This is seen in the existence of \emph{cospectral graphs}, or graphs that are fundamentally different yet yield the same spectrum for a particular matrix. 

By exploring cospectral graphs, we further our understanding of the limitations of each type of matrix. One of the most well-known constructions of cospectral graphs for the adjacency matrix is Godsil-McKay switching. This is done by defining specific subsets of the vertices of a particular graph and constructing a cospectral mate by exchanging edges and non-edges between these subsets. Godsil and McKay \cite{GM} prove the adjacency matrices of two graphs related by this edge switching are similar, and therefore the graphs are cospectral. 

In this paper, we consider cospectral graphs for the \emph{distance matrix}. The distance matrix $D^{(G)}=\left[d_{ij}^{(G)}\right]$ of a connected graph $G=(V(G),E(G))$ is a symmetric matrix such that $d_{ij}^{(G)}$ is the distance, or length of the shortest path, between vertices $i$ and $j$. Its multiset of eigenvalues is the \emph{distance spectrum} of $G$ and two graphs are considered to be distance cospectral if their distance spectra are the same. There has been extensive work done on the distance spectra of graphs (see \cite{survey} for a survey of recent results). 

However, relatively little is known in regard to distance cospectral pairs. McKay \cite{MK} gives a construction for distance cospectral trees by considering any rooted tree and identifying the root with the root of one of two particular trees. Further, he proves the complement graphs of trees constructed in this fashion are also distance cospectral.  Both proofs rely on manipulation of the distance characteristic polynomial. This is the only known distance cospectral graph construction in the literature, and we note that pairs constructed in this manner must contain the same number of edges. In particular, prior to this paper it was not known whether a family could be constructed where distance cospectral pairs could have differing numbers of edges.  

In this paper, we give a construction for distance cospectral graphs with differing numbers of edges in Section~\ref{sec:diff}, and in Section~\ref{sec:switch} we describe a local edge switching behavior which produces more distance cospectral graphs. While these distance switching pairs do not differ in number of edges, they do account for all distance cospectral pairs on seven vertices (see Figure~\ref{sevens}). Finally, in Section~\ref{sec:conc}, we consider further questions of interest for the distance matrix. We complete the introduction with an elementary discussion of graph identification, a process which will be used in subsequent sections.

\subsection{Graph Identification}
\label{ssec:GI}
Throughout our constructions, we will frequently make use of \emph{graph identification}, therefore we define it here and state some observations about distances between vertices in graphs formed in this way.  Let $G,K$ be graphs and let $u \in V(G)$, $v \in V(K)$.  We construct the graph $GK(u,v)$ by identifying the vertices $u$ and $v$ into a new vertex $uv$ in the graph $G\cup K$. When clear context allows, we will denote this graph $GK$. 

Consider calculating the distance between two vertices $x,y$ of $GK$. We can easily do this by considering if $x$ and $y$ are in the $G$ portion of $GK$ or the $K$ portion of $GK$. 
\begin{itemize}
\item If $x,y$ are both in the $G$ portion, $d^{(GK)}_{xy}=d^{(G)}_{xy}$.
\item If $x,y$ are both in the $K$ portion, $d^{(GK)}_{xy}=d^{(K)}_{xy}$.
\item If $x$ is in the $G$ portion and $y$ is in the $K$ portion, $d^{(GK)}_{xy}=d^{(G)}_{xu}+d^{(K)}_{vy}$.
\end{itemize}
These claims can be verified by noticing that a shortest path between vertices in the same portion will be fully contained in that portion. Further, if two vertices are not in the same portion, any path between them must include the vertex $uv$. 

\section{Distance cospectral graphs with differing numbers of edges}
\label{sec:diff}
Consider the two graphs $G$ and $H$ are shown below, each with vertices labeled zero through nine. 
\begin{figure}[h]
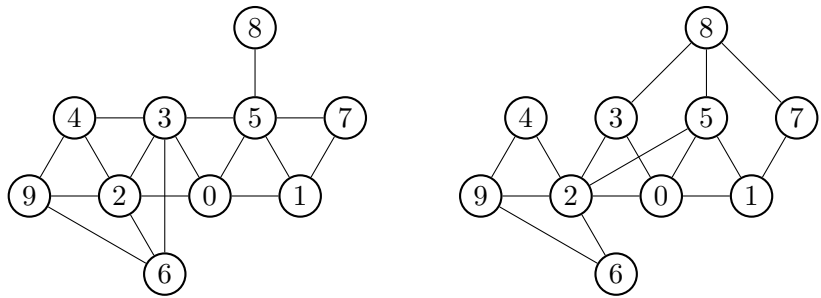

\centering
\tenspic
\caption{Graphs $G$ (left) and $H$ (right).}
\label{GH}
\end{figure}

We immediately note $G$ has 17 edges and $H$ has 16 edges. For future reference, we give the distance matrices of both graphs below.
\[D^{(G)} =\left(\begin{array}{rrrrrrrrrr}
0 & 1 & 1 & 1 & 2 & 1 & 2 & 2 & 2 &
2 \\
1 & 0 & 2 & 2 & 3 & 1 & 3 & 1 & 2 &3 \\
1 & 2 & 0 & 1 & 1 & 2 & 1 & 3 & 3 &1 \\
1 & 2 & 1 & 0 & 1 & 1 & 1 & 2 & 2 &2 \\
2 & 3 & 1 & 1 & 0 & 2 & 2 & 3 & 3 &1 \\
1 & 1 & 2 & 1 & 2 & 0 & 2 & 1 & 1 &3 \\
2 & 3 & 1 & 1 & 2 & 2 & 0 & 3 & 3 &1 \\
2 & 1 & 3 & 2 & 3 & 1 & 3 & 0 & 2 &4 \\
2 & 2 & 3 & 2 & 3 & 1 & 3 & 2 & 0 &4 \\
2 & 3 & 1 & 2 & 1 & 3 & 1 & 4 & 4 &0
\end{array}\right) D^{(H)}=\left(\begin{array}{rrrrrrrrrr}
0 & 1 & 1 & 1 & 2 & 1 & 2 & 2 & 2 &2 \\
1 & 0 & 2 & 2 & 3 & 1 & 3 & 1 & 2 &3 \\
1 & 2 & 0 & 1 & 1 & 1 & 1 & 3 & 2 &1 \\
1 & 2 & 1 & 0 & 2 & 2 & 2 & 2 & 1 &2 \\
2 & 3 & 1 & 2 & 0 & 2 & 2 & 4 & 3 &1 \\
1 & 1 & 1 & 2 & 2 & 0 & 2 & 2 & 1 &2 \\
2 & 3 & 1 & 2 & 2 & 2 & 0 & 4 & 3 &1 \\
2 & 1 & 3 & 2 & 4 & 2 & 4 & 0 & 1 &4 \\
2 & 2 & 2 & 1 & 3 & 1 & 3 & 1 & 0 &3 \\
2 & 3 & 1 & 2 & 1 & 2 & 1 & 4 & 3 &0
\end{array}\right)\]

\begin{theorem} For any graph $K$ and any vertex $v \in V(K)$, and for $u \in \{0,1\}$, the graphs $GK(u,v)$ and $HK(u,v)$ are distance cospectral. 
\label{thrm:tens}
\end{theorem}
\begin{proof}
When identifying the graph $K$ onto $G$, we will enforce that the vertices will be labeled as follows. The vertex $v$ will have the same label as $u$, and the remaining vertices will be labeled with the set $\{10,11,12,\ldots,n\}$. We similarly label $HK$.  Let $D^{(GK)}$ be the distance matrix of $GK$ and similarly $D^{(HK)}$ for $HK$. The proof given will handle the case where $u=0$. The case where $u=1$ is done similarly. 

Let $(\lambda, x)$ be an eigenpair for $D^{(GK)}$ where $\lambda \neq -\frac{1}{2}$.  We claim the vector $y:=x+\Delta$ is an eigenvector of $D^{(HK)}$ for eigenvalue $\lambda$, where 
\[ \Delta_i = \left\{ \begin{array}{c l} 0& i \in \{0,1,10,11,12,\ldots, n\} \\ \alpha & i \in \{2,9\} \\ -\alpha & i \in \{4,6\} \\ \beta & i \in \{3,7\} \\ -\alpha-\beta & i=5 \\ \alpha-\beta & i = 8 \end{array} \right. \]
where 
\[\alpha = \frac{-x_3-x_5-x_7-x_8}{2\lambda+1} \: \text{ and } \: \beta=\frac{\lambda+1}{2\lambda+1}\left(x_5+x_8\right)-\frac{\lambda}{2\lambda+1}\left(x_3+x_7\right).\]

To prove this, we will consider $(D^{(HK)}y)_i$ for all $i$. By inspection of the two matrices, $d_{ij}^{(H)}=d_{ij}^{(G)}$ for all $j \in \{0,1,2,\ldots 9\}$ for $i \in \{0,1\}$. A straightforward algebraic substitution and simplification proves $(D^{(HK)}y)_i=\lambda y_i$ for $i \in \{0,1\}$.
Consider $i \in \{10,11,\ldots,n\}$. We will fully elaborate the steps taken in the following work, as similar processes will be repeated frequently. 
\[ (D^{(HK)}y)_i = \sum_{j=0}^n d_{ij}^{(HK)} y_j = \sum_{j=0}^9  d_{ij}^{(HK)} (x_j+\Delta_j) +\sum_{j=10}^n  d_{ij}^{(HK)} (x_j+\Delta_j) \]
We immediately break the summation into the first ten vertices and the rest, as we will need to treat each group separately. We also substitute the definition of $y$. Next, we use the observations from Section~\ref{ssec:GI} to break the distances in $HK$ to distances in $H$ and $K$, recalling that $v$ is the vertex in the graph $K$ we identify with $0$ in $H$ to create $HK$.  
\begin{align*}
&= \sum_{j=0}^9  (d_{iv}^{(K)}+d_{0j}^{(H)}) (x_j+\Delta_j) +\sum_{j=10}^n   d_{ij}^{(K)} (x_j+\Delta_j) \\
&= \sum_{j=0}^9  (d_{iv}^{(K)}+d_{0j}^{(G)})(x_j+\Delta_j) +\sum_{j=10}^n   d_{ij}^{(K)} (x_j+\Delta_j) 
\end{align*}
We can substitute $d_{0j}^{(G)}$ for $d_{0j}^{(H)}$ by inspection of the first rows of the matrices $D^{(G)}$ and $D^{(H)}$.  We continue by regrouping terms and recombining sums of distances in $G$ and $K$ to be distances in $GK$, again by the observations from Section~\ref{ssec:GI}.
\begin{align*}
&=\sum_{j=0}^9  (d_{iv}^{(K)}+d_{0j}^{(G)})x_j+\sum_{j=10}^n   d_{ij}^{(K)} x_j+\sum_{j=0}^9  (d_{iv}^{(K)}+d_{0j}^{(G)})\Delta_j  \\
&=\sum_{j=0}^n  d_{ij}^{(GK)} x_j+\sum_{j=0}^9  (d_{iv}^{(K)}+d_{0j}^{(G)})\Delta_j  \\
&= \lambda x_i+\alpha(3d_{iv}^{(K)}-3d_{iv}^{(K)}+d_{02}^{(G)}-d_{04}^{(G)}-d_{05}^{(G)}-d_{06}^{(G)}+d_{08}^{(G)}+d_{09}^{(G)}) 
\\ & \hspace{10 pt} +\beta(2d_{iv}^{(K)}-2d_{iv}^{(K)}+d_{03}^{(G)}-d_{05}^{(G)}+d_{07}^{(G)}-d_{08}^{(G)}) \\
&=\lambda x_i = \lambda y_i.
\end{align*}
The last few steps result from the fact that $(\lambda,x)$ is an eigenpair for $D^{(GK)}$ and by direct computation and substitution. This proves $(D^{(HK)}y)_i=\lambda y_i$ for $i \in \{10,11,\ldots, n\}$. 

We now consider the vertices $\{2,3,\ldots,8\}$ by showing the case where $i=2$ and considering how the work generalizes. For this case, notice that $d_{2j}^{(H)}=d_{2j}^{(G)}$ for $j\not \in \{5,8\}$ and $d_{2j}^{(H)}=d_{2j}^{(G)}-1$ for $j \in \{5,8\}$.  
\begin{align*}
(D^{(HK)}y)_2 &= \sum_{j=0}^n d_{2j}^{(HK)} y_j \\
&= \sum_{j=0}^9  d_{2j}^{(HK)} (x_j+\Delta_j) +\sum_{j=10}^n  d_{2j}^{(HK)} (x_j+\Delta_j) \\
&= \sum_{j=0}^9  d_{2j}^{(H)} (x_j+\Delta_j) +\sum_{j=10}^n  (d_{20}^{(H)}+ d_{vj}^{(K)}) (x_j+\Delta_j) \\
&= \sum_{\substack{j=0 \\ j\neq 5,8}}^9  d_{2j}^{(G)} (x_j+\Delta_j)+(d_{25}^{(G)}-1)(x_5+\Delta_5) \\
&\hspace{20pt}+(d_{28}^{(G)}-1)(x_8+\Delta_8) +\sum_{j=10}^n  (d_{20}^{(G)}+d_{vj}^{(K)}) (x_j+\Delta_j) \\
&= \sum_{j=0}^9 d_{2j}^{(G)} x_j+\sum_{j=10}^n  (d_{20}^{(G)}+d_{vj}^{(K)}) x_j+\sum_{j=0}^9 d^{(G)}_{2j}\Delta_j -x_5-x_8-\Delta_5-\Delta_8 \\
&= \sum_{j=0}^n d_{2j}^{(GK)} x_j+\sum_{j=0}^9 d^{(G)}_{2j}\Delta_j -x_5-x_8-\Delta_5-\Delta_8 \\
&= \lambda x_2 +\alpha(d_{22}^{(G)}-d_{24}^{(G)}-d_{25}^{(G)}-d_{26}^{(G)}+d_{28}^{(G)}+d_{29}^{(G)}) 
\\ & \hspace{20 pt} +\beta(d_{23}^{(G)}-d_{25}^{(G)}+d_{27}^{(G)}-d_{28}^{(G)})-x_5-x_8-\Delta_5-\Delta_8 \\
&= \lambda x_2-\beta-x_5-x_8-(-\alpha-\beta)-(\alpha-\beta) \\
&= \lambda (x_2+\alpha)-\lambda\alpha+\beta-x_5-x_8 \\
&= \lambda y_2-\lambda\alpha+\beta-x_5-x_8 \\
\end{align*}

Let $c_2$ be the ``remainder'' terms, specifically $c_2:=-\lambda\alpha+\beta-x_5-x_8$.  To finish the claim that $(D^{(HK)}y)_2=\lambda y_2$, it would suffice to show $c_2=0$:
\begin{align*}-\lambda \alpha +\beta-x_5-x_8 &= -\lambda\left(\frac{-x_3-x_5-x_7-x_8}{2\lambda+1}\right)-x_5-x_8\\
&\hspace{20pt} +\left(\frac{\lambda+1}{2\lambda+1}\left(x_5+x_8\right)-\frac{\lambda}{2\lambda+1}\left(x_3+x_7\right)\right) \\ 
&= (x_3+x_7)\left(\frac{\lambda}{2\lambda+1}-\frac{\lambda}{2\lambda+1}\right)+(x_5+x_8)\left(\frac{\lambda}{2\lambda+1}-1+\frac{\lambda+1}{2\lambda+1}\right) \\ &=0.
\end{align*}

Therefore, by the definition of $\alpha$ and $\beta$, the claim holds for $i=2$.  Repeating this process, we calculate the remainder terms $c_i$ for $i \in \{3,4,\ldots,8\}$ (meaning $(D^{(HK)}y)_i=\lambda y_i+c_i$ for all $i$) in a similar fashion. These are listed below.
\begin{align*}
c_2 &= -\lambda\alpha+\beta-x_5-x_8 \\
c_3 &= -\lambda\beta -2\alpha-\beta+x_4+x_5+x_6-x_8 \\
c_4 &= \lambda\alpha + \alpha+\beta +x_3+x_7 \\
c_5 &= \lambda\alpha+\lambda\beta +3\beta-x_2+x_3+x_7-x_9 \\
c_6 &= \lambda\alpha + \alpha+\beta +x_3+x_7 \\
c_7 &= -\lambda\beta -2\alpha-\beta+x_4+x_5+x_6-x_8 \\
c_8 &= -\lambda \alpha+\lambda\beta -2\alpha+\beta-x_2-x_3-x_7-x_9 \\
c_9 &= -\lambda\alpha+\beta-x_5-x_8
\end{align*}

Similarly to the case where $i=2$, our goal is to show that all remaining $c_i$ are equal to zero. Substitution of $\alpha$ and $\beta$ suffices for $c_4$.
To prove $c_3$ and $c_5$, we consider combinations of particular rows of $D^{(GK)}$. We claim the following three equations hold:
\begin{equation}
2x_3+x_4+2x_5+x_6+2x_7=\lambda\left(x_2-x_4-x_5-x_6+x_8+x_9\right),
\label{rowsum1}
\end{equation}
\begin{equation}
x_2-x_3-3x_5-x_7-3x_8+x_9 = \lambda\left(-x_2-x_3+x_4+2x_5+x_6-x_7-x_9\right),
\label{rowsum2}
\end{equation}
and
\begin{equation}
x_2+x_3+x_4-x_5+x_6+x_7-3x_8+x_9 = \lambda\left(-x_3-x_7+x_5+x_8\right).
\label{rowsum3}
\end{equation}

We will only prove~\eqref{rowsum1}, as this proof can be generalized into proofs for~\eqref{rowsum2} and \eqref{rowsum3}. Let $D_i^{(GK)}$ denote the $i$th row of the matrix $D^{(GK)}$. Further, let $e_i$ be the $i$th standard row vector. Consider the following sum and difference of rows of $D^{(GK)}$
\[m:=D^{(GK)}_2-D^{(GK)}_4-D^{(GK)}_5-D^{(GK)}_6+D^{(GK)}_8+D^{(GK)}_9.\]
By definition, the $i$th entry of $m$ is 
\[m_i=d^{(GK)}_{i2}-d^{(GK)}_{i4}-d^{(GK)}_{i5}-d^{(GK)}_{i6}+d^{(GK)}_{i8}+d^{(GK)}_{i9}.\]
If $i \in \{0,1,\ldots,9\}$, $d_{ij}^{(GK)}=d_{ij}^{(G)}$ for $j \in \{0,1,\ldots,9\}$, therefore the first 10 entries of $m$ can be computed directly from $D^{(G)}$.  Consider $i \in \{10,11,\ldots, n\}$. Recall $d_{ij}^{(GK)}=d_{iv}^{(K)}+d_{0j}^{(G)}$ for $j \in \{0,1,\ldots,9\}$.  In this case, the $i$th entry of $m$ is 
\begin{align*}
m_i &=3d_{iv}^{(K)}-3d_{iv}^{(K)}+d_{02}^{(G)}-d^{(G)}_{04}-d^{(G)}_{05}-d^{(G)}_{06}+d^{(G)}_{08}+d^{(G)}_{09} \\
&= d_{02}^{(G)}-d^{(G)}_{04}-d^{(G)}_{05}-d^{(G)}_{06}+d^{(G)}_{08}+d^{(G)}_{09} \\
&=0
\end{align*}
therefore we can write the following equation
\[D^{(GK)}_2-D^{(GK)}_4-D^{(GK)}_5-D^{(GK)}_6+D^{(GK)}_8+D^{(GK)}_9=2e_3+e_4+2e_5+e_6+2e_7.\]
Because $x$ is an eigenvector of $D^{(GK)}$, $D_i^{(GK)}x=(D^{(GK)}x)_i=\lambda x_i$ for all $i$.  By multiplying by $x$ on both sides of the equation above on the right, we see   
\begin{align*}
\left(D^{(GK)}_2-D^{(GK)}_4-D^{(GK)}_5-D^{(GK)}_6+D^{(GK)}_8+D^{(GK)}_9\right)x&=\left(2e_3+e_4+2e_5+e_6+2e_7\right)x \\
\lambda(x_2-x_4-x_5-x_6+x_8+x_9)&=2x_3+2x_5+2x_7+x_4+x_6
\end{align*}
which is~\eqref{rowsum1}.  Equations~\eqref{rowsum2} and~\eqref{rowsum3} follow similarly by considering appropriate row combinations. 

With these three equations, we can prove $c_3$ and $c_5$ are zero. We begin the work for $c_3$ by substituting the definitions of $\alpha$ and $\beta$:
\begin{align*} 
c_3 &= -\lambda\beta -2\alpha-\beta+x_4+x_5+x_6-x_8 \\
&=(-\lambda-1)\left(\frac{\lambda+1}{2\lambda+1}\left(x_5+x_8\right)-\frac{\lambda}{2\lambda+1}\left(x_3+x_7\right)\right) \\& \hspace{20 pt} -2\left(\frac{-x_3-x_5-x_7-x_8}{2\lambda+1} \right)+x_4+x_5+x_6-x_8 \\
&= \frac{\lambda^{2} {\left(x_{3} - x_{5} + x_{7} - x_{8}\right)} + \lambda{\left(x_{3} + 2 x_{4} + 2x_{6} + x_{7} -  4x_{8}\right)}}{2 \lambda + 1} \\& \hspace{20 pt} + \frac{ 2 x_{3} + x_{4} + 2 x_{5} + x_{6} + 2 x_{7}}{2\lambda+1} 
\end{align*}
Consider the last term above.  The numerator is the left hand side of~\eqref{rowsum1}, and we can substitute the right hand side. 
\begin{align*}
&= \frac{\lambda^{2} {\left(x_{3} - x_{5} + x_{7} - x_{8}\right)} + \lambda{\left(x_{3} + 2 x_{4} + 2x_{6} + x_{7} -  4x_{8}\right)}}{2 \lambda + 1} \\& \hspace{20 pt} + \frac{\lambda\left(x_2-x_4-x_5-x_6+x_8+x_9\right)}{2\lambda+1} \\ 
&= \frac{\lambda^{2} {\left(x_{3} - x_{5} + x_{7} - x_{8}\right)} + \lambda{\left(x_2+x_{3} +  x_{4}-x_5 + x_{6} + x_{7} -  3x_{8}+x_9\right)}}{2 \lambda + 1}
\end{align*}
Here we see the linear combination of terms that is multiplied by $\lambda$ is the left hand side of~\eqref{rowsum3}. Similarly to before, we will substitute the right hand side and cancel. 
\begin{align*}
&= \frac{\lambda^{2} {\left(x_{3} - x_{5} + x_{7} - x_{8}\right)} + \lambda\left({\lambda\left(-x_3-x_7+x_5+x_8\right)}\right)}{2 \lambda + 1} =0
\end{align*}
This proves $c_3=0$. A similar substitution of~\eqref{rowsum2} and \eqref{rowsum3} yield $c_5=0$. Finally, notice $c_8=c_5-2c_6$, and thus $c_8$ is also zero. This validates the claim that $y$ is an eigenvector of $D^{(HK)}$.

We note the mapping of eigenpairs of $D^{(GK)}$ where $\lambda\neq -\frac{1}{2}$ to those of $D^{(HK)}$ where $\lambda\neq -\frac{1}{2}$ is injective. Suppose $(\lambda,x),(\lambda,x')$ are eigenpairs of $D^{(GK)}$ such that $y=y'$, or equivalently $x+\Delta=x'+\Delta'$. We will show $\Delta=\Delta'$ by showing $\alpha=\alpha'$ and $\beta=\beta'$.  
\begin{align*}
y_3+y_5+y_7+y_8 &= y'_3+y'_5+y'_7+y'_8 \\
x_3+x_5+x_7+x_8+2\alpha-2\alpha+2\beta-2\beta &= x'_3+x'_5+x'_7+x'_8+2\alpha'-2\alpha'+2\beta'-2\beta' \\
x_3+x_5+x_7+x_8 &= x'_3+x'_5+x'_7+x'_8 \\
\frac{x_3+x_5+x_7+x_8}{2\lambda+1} &= \frac{x'_3+x'_5+x'_7+x'_8}{2\lambda+1} \\ 
-\alpha &=-\alpha' \\
\alpha &= \alpha'
\end{align*}
To prove $\beta=\beta'$, we recall that 
\[-c_2=\lambda\alpha-\beta+x_5+x_8=0\]
therefore, since $\lambda \alpha= \lambda\alpha'$, 
\begin{align*}
\lambda\alpha-\beta+x_5+x_8 &=\lambda\alpha'-\beta'+x_5'+x_8' \\
-\beta+x_5+x_8 &=-\beta'+x_5'+x_8' \\
-\beta+x_5-\alpha-\beta+x_8+\alpha-\beta+2\beta &=-\beta'+x_5'-\alpha'-\beta'+x_8'+\alpha'-\beta'+2\beta' \\
y_5+y_8+\beta &=y_5'+y_8'+\beta' \\
\beta &= \beta'. 
\end{align*}
Therefore the mapping is injective as claimed. Further, the mapping is also surjective, as we could have started with the graph $HK$ and performed the perturbation in reverse to get eigenpairs of $D^{(GK)}$. 

What remains to be considered are eigenpairs where $\lambda=-\frac{1}{2}$, if any exist.  However, since the map is bijective where defined, the dimensions of the eigenspaces for all eigenvalues not equal to $-\frac{1}{2}$ must be the same for both $D^{(GK)}$ and $D^{(HK)}$. Because the sum of the dimensions of all eigenspaces must be $n$, the multiplicity of $-\frac{1}{2}$ as an eigenvalue must be the same for both $D^{(GK)}$ and $D^{(HK)}$. Therefore the dimensions of all eigenspaces are the same, and the graphs $GK$ and $HK$ are distance cospectral as claimed. 
\end{proof}

We note that the theorem yields a construction for large distance cospectral families with a variety of edge counts. Consider identifying $k$ copies of $G$ at a single vertex, namely vertex $0$ of each copy. By repeated applications of the theorem, we can exchange out copies of $G$ with copies of $H$ one at a time. Doing this, we construct $k+1$ graphs which are mutually distance cospectral and with edge counts $\{16k,16k+1,\ldots,17k\}$. 

\section{Distance switching}
\label{sec:switch}
The proof in Section~\ref{sec:diff} relied on a perturbation of the distance eigenvectors of one graph to yield the distance eigenvectors of another. In this section, we explore a similar technique when considering pairs of distance cospectral graphs related by restricted edge switching.  Suppose a graph $G$ has the following two properties. First, $G$ has one of the graphs in Figure~\ref{can} as an induced subgraph.
\begin{figure}[htb]
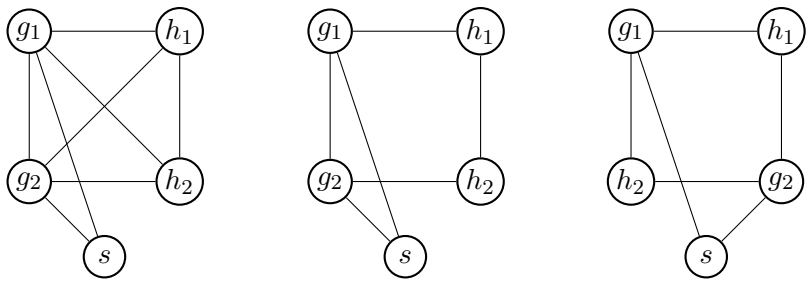

\centering
\quintetspic
\caption{Subgraph switching candidates}
\label{can}
\end{figure}

Second, we can partition the vertices in $V(G) \setminus \{g_1,g_2,h_1,h_2\}$ into two sets, $A$ and $B$, such that for all $v \in A$  
\[ d_{vg_1}^{(G)}+d_{vg_2}^{(G)}-d_{vh_1}^{(G)}-d_{vh_2}^{(G)}=-2,\]

and for all vertices $v \in B$
\[ d_{vg_1}^{(G)}+d_{vg_2}^{(G)}-d_{vh_1}^{(G)}-d_{vh_2}^{(G)}=0.\]

We construct a new graph $H$ as follows. Let $V(H)=V(G)$, and 
\[E(H)=E(G)\setminus \{(s,g_1),(s,g_2)\} \cup \{(s,h_1),(s,h_2)\}.\]


We note this switching is somewhat similar to Godsil-McKay switching. Godsil and McKay's construction for local switching requires a switching set $D$ and a partition of the remaining vertices into sets $\{C_i\}$ where for every vertex in $v \in D$ and every set $C_i$, $v$ is either adjacent to all, none, or exactly half of the vertices in $C_i$. The switching is done by exchanging edges for non-edges between $D$ and the sets $C_i$ where the vertices in $D$ are adjacent to half of the vertices in $C_i$. See Section 2.1 of~\cite{GM} for a full explanation of the construction, including further requirements on the sets $C_i$ not stated here. If we consider the switching set $D$ to be the singleton $s$ and one of the $C_i$ of the partition to be $\{g_1,g_2,h_1,h_2\}$, the construction of $H$ can be likened to Godsil and McKay's construction. 

Because $V(H)=V(G)$ and because we will be referencing distances between vertices in both $G$ and $H$, we will frequently reference the vertex set as simply $V$. 

\begin{theorem} If for all $v \in B$, $d_{vu}^{(H)}=d_{vu}^{(G)}$ for all $u \in V$ and if for all $w \in A$, $d_{wu}^{(H)} = d_{wu}^{(G)}$ for all $u \in V\setminus \{g_1,g_2,h_1,h_2\}$ and 
\[d_{wg_i}^{(H)}=d_{wg_i}^{(G)}+ 1 \: \text{ and } \: d_{wh_i}^{(H)}=d_{wh_i}^{(G)}- 1 \]
for $i \in \{1,2\}$, then $G$ and $H$ are distance cospectral. 
\label{thrm:switching}
\end{theorem}

\begin{proof}

We first define a function $c$ on the vertices to be 
\[c(v)= d_{vg_1}^{(G)}+d_{vg_2}^{(G)}-d_{vh_1}^{(G)}-d_{vh_2}^{(G)}.\]
By our assumptions on $G$ and direct computation, we can establish that
\[c(v) =\left\{ \begin{array}{r l} -2 & v \in A \\ 0 & v \in B \\ -k & v \in \{g_1,g_2\} \\ k & v \in \{h_1,h_2\} \end{array} \right. \]
where $k=1$ for the subgraph on the left in Figure~\ref{can}, $k=2$ for the subgraph in the middle of Figure~\ref{can}, and $k=0$ for the subgraph on the right of Figure~\ref{can}.

Suppose $(\lambda,x)$ is an eigenpair for the matrix $D^{(G)}$ for $\lambda\neq -k$. We claim $y:=x+\Delta$ is an eigenvector of $D^{(H)}$ for eigenvalue $\lambda$, where 
\[ \Delta_i=\left\{
\begingroup \renewcommand*{\arraystretch}{1.5}
 \begin{array}{c l} 0& i \not \in \{g_1,g_2,h_1,h_2\} \\ \frac{\sum_{j \in A}x_j}{\lambda+k}& i \in \{g_1,g_2\} \\  -\frac{\sum_{j \in A}x_j}{\lambda+k}& i \in \{h_1,h_2\}.  \end{array} 
 \endgroup  \right. \]

To prove $y$ is indeed an eigenvector of $D^{(H)}$, we will show $(D^{(H)}y)_i=\lambda y_i$ for each vertex $i$.  First suppose $i \in B$. We immediately note that $d_{iv}^{(G)}=d_{iv}^{(H)}$ for all $v \in V$ by the hypotheses of the theorem. Further, recall $c(i)=0$ for all $i \in B$. We therefore have
\begin{align*}
(D^{(H)} y)_i & = \sum_{j \in V} d_{ij}^{(H)} y_j \\  
&= \sum_{j \in V} d_{ij}^{(G)} (x_j+\Delta_j) \\
&= \sum_{j \in V} d_{ij}^{(G)}x_j + \frac{\sum_{j \in A}x_j}{\lambda+k}\left(d_{ig_1}^{(G)}+d_{ig_2}^{(G)}-d_{ih_1}^{(G)}-d_{ih_2}^{(G)}\right) \\
&= \sum_{j \in V} d_{ij}^{(G)}x_j + \frac{\sum_{j \in A}x_j}{\lambda+k}c(i) \\
&= \sum_{j \in V} d_{ij}^{(G)}x_j \\
& =\lambda x_i =\lambda y_i.
\end{align*}

Now suppose $i \in \{g_1,g_2\}$. We know that for all vertices $v \in B$, $d_{iv}^{(H)}=d_{iv}^{(G)}$. Further, for all vertices $u \in A$,  $d_{ui}^{(H)}=d_{ui}^{(G)}+ 1$ and $c(i)=-k$ for $i \in \{g_1,g_2\}$. Combining these facts, we see that

\begin{align*}
(D^{(H)} y )_i & = \sum_{j \in V} d_{ij}^{(H)} y_j \\ 
&= \sum_{j \in A} (d_{ij}^{(G)}+1) (x_j+\Delta_j)  +\sum_{j \in B} d_{ij}^{(G)} (x_j+\Delta_j) + \sum_{j \in \{g_1,g_2,h_1,h_2\} }d_{ij}^{(G)} (x_j+\Delta_j)  \\
&= \sum_{j \in V}d^{(G)}_{ij}x_j +\sum_{j \in A} x_j+\sum_{j \in \{g_1,g_2,h_1,h_2\}} d^{(G)}_{ij}\Delta_j \\
& =\sum_{j \in V} d_{ij}^{(G)}x_j +\sum_{j \in A}x_j + \frac{\sum_{j \in A}x_j}{\lambda+k}\left(d_{ig_1}^{(G)}+d_{ig_2}^{(G)}-d_{ih_1}^{(G)}-d_{ih_2}^{(G)}\right) \\ 
& =\sum_{j \in V} d_{ij}^{(G)}x_j + \sum_{j \in A}x_j + \frac{\sum_{j \in A}x_j}{\lambda+k}c(i) \\ 
&=\lambda x_i + \sum_{j \in A}x_j -k \frac{\sum_{j \in A}x_j}{\lambda+k} \\ 
& =\lambda x_i +  \frac{\lambda \sum_{j \in A}x_j}{\lambda+k} \\
&= \lambda \left(x_i+\frac{\sum_{j \in A}x_j}{\lambda+k}\right) = \lambda y_i. \\
\end{align*}

A similar algebraic computation suffices for the case $i \in \{h_1,h_2\}$. 
What remains to be checked are the vertices $i \in A$.  We know $d_{iu}^{(H)} = d_{iu}^{(G)}$ for all $u \in V\setminus \{g_1,g_2,h_1,h_2\}$. Further, $d_{ig_\ell}^{(H)}=d_{ig_\ell}^{(G)}+ 1$ and  $d_{ih_\ell}^{(H)}=d_{ih_\ell}^{(G)}- 1$ for $\ell \in \{1,2\}$. Finally, recall that $c(i)=-2$ for all $i \in A$. Therefore 

\begin{align*}
(D^{(H)} y)_i & = \sum_{j \in V} d_{ij}^{(H)} y_j \\ 
&= \sum_{\substack{j \in V \\ j \not\in \{g_1,g_2,h_1,h_2\}}} d_{ij}^{(G)} (x_j+\Delta_j) +\sum_{j \in \{g_1,g_2\}}(d_{ij}^{(G)}+1)(x_j+\Delta_j)+\sum_{j \in \{h_1,h_2\}}(d_{ij}^{(G)}-1)(x_j+\Delta_j) \\
 &= \sum_{j \in V} d_{ij}^{(G)} x_j + x_{g_1}+x_{g_2}-x_{h_1}-x_{h_2} \\
 & \hspace{1 in}+\frac{\sum_{j \in A}x_j}{\lambda+k}\left( d_{ig_1}^{(G)}+1+d_{ig_2}^{(G)}+1-(d_{ih_1}^{(G)}-1)-(d_{ih_2}^{(G)}-1)\right)  \\
&= \sum_{j \in V} d_{ij}^{(G)} x_j + x_{g_1}+x_{g_2}-x_{h_1}-x_{h_2} +\frac{\sum_{j \in A}x_j}{\lambda+k}\left( d_{ig_1}^{(G)}+d_{ig_2}^{(G)}-d_{ih_1}^{(G)}-d_{ih_2}^{(G)}+4\right)  \\
&= \sum_{j \in V} d_{ij}^{(G)} x_j + x_{g_1}+x_{g_2}-x_{h_1}-x_{h_2} +\frac{\sum_{j \in A}x_j}{\lambda+k}\left( c(i)+4\right)  \\
&= \sum_{j \in V} d_{ij}^{(G)} x_j + x_{g_1}+x_{g_2}-x_{h_1}-x_{h_2}+\frac{2\sum_{j \in A}x_j}{\lambda+k}.
\end{align*}

We pause here to prove the following equality:
\[ x_{g_1}+x_{g_2}-x_{h_1}-x_{h_2}+\frac{2\sum_{j \in A}x_j}{\lambda+k} =0.\]
To do this, let $D_i^{(G)}$ denote the $i$th row of the matrix $D^{(G)}$ and $e_{i}$ denote the $i$th standard row vector. We claim 
\[D_{g_1}^{(G)}+D_{g_2}^{(G)}-D_{h_1}^{(G)}-D_{h_2}^{(G)} = -2\sum_{j \in A} e_{j}-ke_{g_1}-ke_{g_2}+ke_{h_1}+ke_{h_2}.\]
Suppose $m :=D_{g_1}^{(G)}+D_{g_2}^{(G)}-D_{h_1}^{(G)}-D_{h_2}^{(G)}$, and consider the $j$th entry of $m$:
\[m_j = d_{jg_1}^{(G)}+d_{jg_2}^{(G)}-d_{jh_1}^{(G)}+d_{jh_2}^{(G)}.\] 
This by definition is $c(j)$, and the claim follows. 

With this in mind, we multiply both sides by $x$ on the right. Because $x$ is an eigenvector, we know $D_{j}^{(G)}x=(D^{(G)}x)_j=\lambda x_{j}$ for all $j$. Therefore we have
\begin{align*}
\left(D_{g_1}^{(G)}+D_{g_2}^{(G)}-D_{h_1}^{(G)}-D_{h_2}^{(G)} \right)x &= \left(-2\sum_{j \in A}x_j-ke_{g_1}-ke_{g_2}+ke_{h_1}+ke_{h_2}\right)x \\ 
\lambda x_{g_1}+\lambda x_{g_2}-\lambda x_{h_1}-\lambda x_{h_2} &=-2 \sum_{j \in A}x_j-kx_{g_1}-kx_{g_2}+kx_{h_1}+kx_{h_2} \\
(\lambda+k)(x_{g_1}+x_{g_2}-x_{h_1}-x_{h_2})&=-2\sum_{j \in A}x_j \\
x_{g_1}+x_{g_2}-x_{h_1}-x_{h_2} &= \frac{-2\sum_{j \in A}x_j}{\lambda+k}
\end{align*}
which proves the equality. Returning to our case, 
\begin{align*}
(D^{(H)} y)_i & = \sum_{j \in V} d_{sj}^{(G)} x_j + x_{g_1}+x_{g_2}-x_{h_1}-x_{h_2}+\frac{2\sum_{j \in A}x_j}{\lambda+k} \\ &=\lambda x_i = \lambda y_i
\end{align*}
which finishes the case for $i \in A$.  Thus the vector $y$ is an eigenvector as claimed. 

We note that this mapping of eigenpairs of $D^{(G)}$ with $\lambda\neq -k$ to eigenpairs of $D^{(H)}$ for $\lambda \neq -k$ is bijective. Suppose there are two distinct eigenvectors $x$ and $x'$ for $D^{(G)}$ with the same eigenvalue $\lambda \neq -k$ that map to the same eigenvector $y$ for $D^{(H)}$. Then $y_i=x_i=x_i'$ for all $i \not \in \{g_1,g_2,h_1,h_2\}$. If $i \in \{g_1,g_2\}$, then for all $j \in A$,
\begin{align*}
x_i+\frac{\sum_{j \in A}x_j}{\lambda+k} =  y_i &= x'_i+\frac{\sum_{j \in A}x'_j}{\lambda+k} \\
x_i+\frac{\sum_{j \in A}x_j}{\lambda+k} = & \: x'_i+\frac{\sum_{j \in A}x_j}{\lambda+k} \\ 
x_i  =  &\: x'_i 
\end{align*}
and similarly if $i \in \{h_1,h_2\}$. This implies $x=x'$, and the mapping is injective. Certainly the map is also surjective because we could have perturbed instead the eigenvectors of $H$. Because of this, we notice that the dimensions of all eigenspaces are the same for $D^{(G)}$ and $D^{(H)}$ for $\lambda\neq -k$. 

What remains to be considered are eigenpairs $(\lambda,x)$ for $\lambda=-k$, if any exist. However, since the mapping is bijective where defined, the dimensions of the eigenspaces for all eigenvalues not equal to $-k$ are the same for both $D^{(G)}$ and $D^{(H)}$. Because the sum of all eigenspaces must be the order of $G$, the multiplicity of $-k$ must be the same for both graphs. Thus the graphs $G$ and $H$ are distance cospectral as claimed. 
\end{proof}

We note that while this edge switching behavior may seem restrictive, it does explain all pairs of distance cospectral graphs on seven vertices, checked by exhaustive search.  Figure~\ref{sevens} shows these graphs, arranged in the table in order of the induced subgraph contained from Figure~\ref{can}. We point out that the vertex $s$ is shown at the bottom of every embedding.
\begin{figure}[h!]
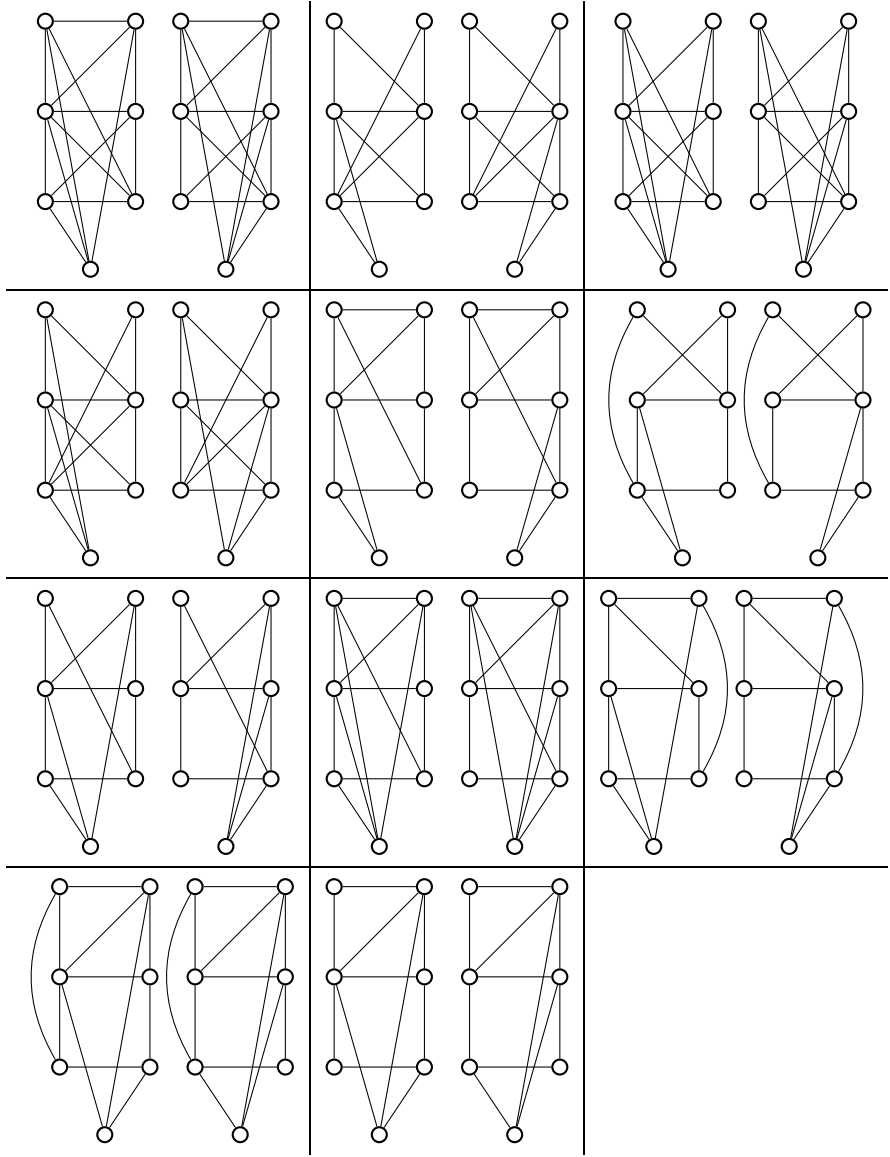

\centering
\sevens
\caption{All distance cospectral pairs on seven vertices.}
\label{sevens}
\end{figure} 

Further, once we have found a pair $G,H$ that follows this switching behavior, we claim that we can construct infinitely many more pairs using graph identification. 

\begin{corollary} Let $G, H$ be a distance cospectral pair of graphs given by Theorem~\ref{thrm:switching}, and let $u \in V(G)\setminus \{g_1,g_2,h_1,h_2\}$. For any graph $K$ and any vertex $v \in V(K)$, the graphs $GK(u,v), HK(u,v)$ are distance cospectral.
\end{corollary}
\begin{proof}
We first require some notation. Let $A_G$ and $B_G$ be the partition of $V(G)\setminus \{g_1,g_2,h_1,h_2\}$ given by the construction preceding Theorem~\ref{thrm:switching}. 

We need only to show that this new pair of graphs satisfies the original switching construction and the hypotheses of Theorem~\ref{thrm:switching}. Certainly $GK$ contains one of the induced subgraphs in Figure~\ref{can} because $G$ does. 

To construct the partition $A,B$ of $GK$, we will extend the partition $A_G,B_G$ in a predicable way. Notice first that for any vertex $x$ in the $G$ portion of $GK$ and for any vertex $w \in \{g_1,g_2,h_1,h_2\}$, 
\[d_{xw}^{(GK)}=d_{xw}^{(G)}\] 
by the construction of the graph.  Therefore, if $x \in A_G$, it follows that $x \in A$ for $GK$, and similarly for $x \in B_G$. 

We now aim to partition the vertices in the $K$ portion of $GK$. Suppose $w$ is such a vertex. For any vertex $x$ in the $G$ portion of $GK$, by Section~\ref{ssec:GI}, we know 
\[d_{wx}^{(GK)} = d_{wv}^{(K)}+d_{ux}^{(G)}.\]
This implies 
\begin{align*}
d_{wg_1}^{(GK)}+d_{wg_2}^{(GK)}-d_{wh_1}^{(GK)}-d_{wh_2}^{(GK)} &=d_{wv}^{(K)}+d_{ug_1}^{(G)}+d_{wv}^{(K)}+d_{ug_2}^{(G)} -d_{wv}^{(K)}-d_{uh_1}^{(G)}-d_{wv}^{(K)}-d_{uh_2}^{(G)} \\ &= d_{ug_1}^{(G)}+d_{ug_2}^{(G)}-d_{uh_1}^{(G)}-d_{uh_2}^{(G)}
\end{align*}
which is either $-2$ or $0$, depending on if $u$ is in $A_G$ or in $B_G$. Thus the necessary partition holds for all vertices in $GK$. Specifically, if $u \in A_G$, all vertices in the $K$ portion of $GK$ are in $A$. Similarly, if $u \in B_G$, all vertices in the $K$ portion of $GK$ are in $B$.

Certainly the graph $HK$ is the graph which is formed by the switching construction on $GK$. We now must prove $HK$ meets the hypotheses of the theorem. We start by showing that for any two vertices $x,y$ such that neither is in $\{g_1,g_2,h_1,h_2\}$, $d_{xy}^{(GK)}=d_{xy}^{(HK)}$.  Suppose $x,y$ are two such vertices. First we consider if both are in the $K$ portion of $HK$.  Then, by the construction of the graph \[d_{xy}^{(GK)}=d_{xy}^{(K)}=d_{xy}^{(HK)}.\] 
If both are in the $H$ portion of $HK$, then because $H$ met the hypotheses of the theorem applied to the pair $G,H$, we know $d_{xy}^{(GK)}=d_{xy}^{(HK)}$.  Now consider if $x\in H$ and $y \in K$.  We again use the fact that any path between vertices in $H$ and $K$ must pass through the identified vertex, and we can write 
\[d_{xy}^{(HK)}=d_{xu}^{(H)}+d_{vy}^{(K)}\] and we have two instances of the previous cases, where both vertices are in $H$ and $K$.   

We now need to consider the distances between $\{g_1,g_2,h_1,h_2\}$ and the remaining vertices in the graph. Suppose $w \in A$. If $w$ is in the $H$ portion of $HK$, then because $H$ meets the conditions of Theorem~\ref{thrm:switching}, 
\[d_{wg_i}^{(HK)}=d_{wg_i}^{(GK)}+ 1 \: \text{ and } \: d_{wh_i}^{(HK)}=d_{wh_i}^{(GK)}- 1 \]
for $i \in \{1,2\}$. 

If $w$ is in the $K$ portion of $HK$, then we notice that by the extension of $A_G$ and $B_G$ into $A$ and $B$, we know $u \in A_G$.  
This means 
\[d_{ug_i}^{(H)}=d_{ug_i}^{(G)}+ 1 \: \text{ and } \: d_{uh_i}^{(H)}=d_{uh_i}^{(G)}- 1 \]
for $i \in \{1,2\}$. 

We can therefore write 
\[ d_{wg_i}^{(HK)}=d_{wv}^{(K)}+d_{ug_i}^{(H)}=d_{wv}^{(K)}+d_{ug_i}^{(G)}+ 1=d_{wg_i}^{(GK)}+1\]
and
\[ d_{wh_i}^{(HK)}=d_{wv}^{(K)}+d_{uh_i}^{(H)}=d_{wv}^{(K)}+d_{ug_i}^{(G)}- 1=d_{wg_i}^{(GK)}-1\]
for $i \in \{1,2\}$.  

If $w \in B$, we follow a parallel argument and use the fact that $u$ must be in $B_G$. Therefore $GK,HK$ meet the conditions of Theorem~\ref{thrm:switching}, and $GK$ and $HK$ are distance cospectral. 
\end{proof}

\section{Conclusion}
\label{sec:conc}

We have established two constructions for distance cospectral pairs (and indeed, large distance cospectral families), including one where graphs have differing numbers of edges. It is interesting to note that distance cospectral graphs with differing numbers of edges are rare. Other than the graphs show in Section~\ref{sec:diff}, there are only two distance cospectral pairs on ten vertices or fewer. These are shown in Figure~\ref{fig:others}. 

\begin{figure}[htb]
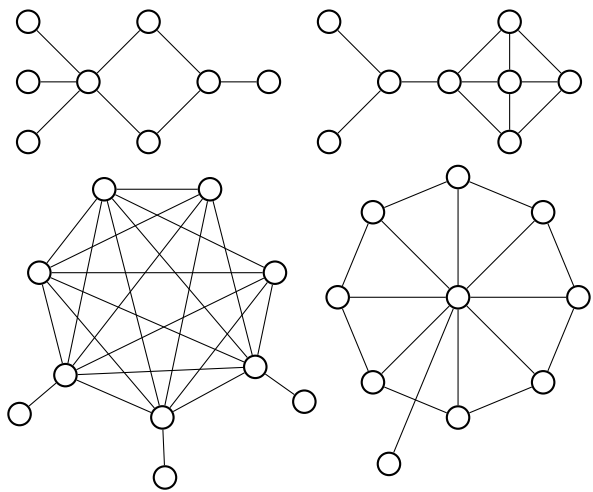

\centering

\othereights

\vspace{3pt} 

\othertens
\caption{Distance cospectral graph pairs with differing numbers of edges.}
\label{fig:others}
\end{figure}

This emphasis on the edge count fits in a larger question of what the spectrum of any matrix can tell about the graph's structure. For well studied matrices, the questions of whether cospectral pairs exist with differing number of components or whether pairs exist where one graph is bipartite and one is not have been answered. Only one of these questions is relevant for the distance matrix, since the distance matrix is not defined for disconnected graphs. It would be interesting to know if distance cospectral pairs exist where one graph is bipartite and the other is not; no such pair exists on ten vertices or fewer. We hope to see exploration of this problem and more work for distance cospectral constructions in the future.

\end{document}